\documentclass[12pt,a4]{amsart}
\usepackage[a4paper, left=28mm, right=28mm, top=28mm, bottom=34mm]{geometry}
\usepackage{amsmath}
\usepackage{amsthm}
\usepackage{bm}
\newtheorem{dfn}{Definition}[section]
\newtheorem{thm}[dfn]{Theorem}
\newtheorem{prop}[dfn]{Proposition}
\newtheorem{lem}[dfn]{Lemma}
\newtheorem{cor}[dfn]{Corollary}

\newtheorem{rem}[dfn]{Remark}
\newtheorem{conj}[dfn]{Conjecture}

\newtheorem{cond}[dfn]{Condition}

\usepackage{amssymb}
\usepackage{amscd}
\usepackage{mathrsfs}

\def\Q{\mathbb{Q}}
\def\R{\mathbb{R}}
\def\Z{\mathbb{Z}}
\def\C{\mathbb{C}}
\def\F{\mathbb{F}}
\def\Br{\mathop{\mathrm{Br}}\nolimits}

\def\Pic{\mathop{\mathrm{Pic}}\nolimits}

\def\M2dh{\mathop{M^{(h)}_{2d}}\nolimits}
\def\M2d{\mathop{M_{2d}}\nolimits}
\def\g1{\mathop{\gamma_1}\nolimits}
\def\g2{\mathop{\gamma_2}\nolimits}
\def\O{\mathop{\mathscr{O}}\nolimits}

\def\disc{{\rm{disc}}}
\title[$K3$ surfaces with $L$-function]{Unconditional construction of $K3$ surfaces over finite fields with given $L$-function in large characteristic}
\address{Department of Mathematics, Faculty of Science, Kyoto University, Kyoto 606-8502, Japan}
\email{kito@math.kyoto-u.ac.jp}
\date{October 8, 2017}
\author{Kazuhiro Ito}

\begin{document}
\maketitle
\begin{abstract}
We give an unconditional construction of $K3$ surfaces over finite fields with given $L$-function, up to finite extensions of the base fields, under some mild restrictions on the characteristic. Previously, such results were obtained by Taelman assuming semistable reduction. The main contribution of this paper is to make Taelman's proof unconditional. We use some results of Nikulin and Bayer-Fluckiger to construct an appropriate complex projective $K3$ surface with $CM$ which admits an elliptic fibration with a section, or an ample line bundle of low degree. Then using Saito's construction of strictly semistable models and applying a slight refinement of Matsumoto's good reduction criterion for $K3$ surfaces, we obtain a desired $K3$ surface over a finite field.     
\end{abstract}
\section{Introduction}\label{intro}
In this paper, we shall give an unconditional construction of $K3$ surfaces over finite fields with given $L$-function, up to finite extensions of the base fields. Previously, Taelman conditionally proved the existence of such $K3$ surfaces assuming a strong version of the existence of semistable reduction for $K3$ surfaces \cite[Theorem 3]{Taelman}. The main contribution of this paper is to make Taelman's proof unconditional.

Fix a prime number $p$ and an integer $m$ with $1\leq{m}\leq10$. Let $q$ be a power of $p$. We fix an embedding $\overline{\Q} \hookrightarrow \overline{\Q}_p$, and let $\nu_{q}\colon \overline{\Q}_p\to \Q\cup\lbrace\infty\rbrace$ be the $p$-adic valuation normalized by ${\nu_q}(q)=1$.
 For a polynomial 
$$Q(T)=\prod_{j}(1-{{\beta}_{j}}T) \in \Q[T]\qquad (\beta_{j}\in\overline{\Q}),$$
we put $$Q_{<0}(T)=\prod_{\nu_{q}({\beta}_{j})<0}(1-{{\beta}_{j}}T) \in \Q_{p}[T].$$

 Consider a polynomial $$L(T) \in{1+T\Q[T]}$$ of degree $2m$ satisfying the following conditions.
\begin{cond}\label{cond}
\begin{itemize}
\item All complex roots of $L(T)$ have absolute value one,
\item no root of $L(T)$ is a root of unity,
\item $L(T) \in \Z_{\ell}[T]$ for all prime numbers ${\ell}\neq{p}$,
\item there exists a positive integer $h\in \Z$ with $1\leq{h}\leq{m}$ such that, if we denote the roots of $L(T)$ by $\alpha_1, \dotsc, \alpha_{2m}\in\overline{\Q}$, then, after permuting them, they satisfy
$$
\begin{cases}
\nu_{q}({\alpha}_i)=-1/h &\quad \  ({1}\leq{i}\leq{h})\\
\nu_{q}({\alpha}_i)=0  &\quad  \  ({h+1}\leq{i}\leq{2m-h})\\
\nu_{q}({\alpha}_i)=1/h  &\quad \ ({2m-h+1}\leq{i}\leq{2m})\\
\end{cases}
$$
for a power $q$ of $p$, 
\item $L(T)=Q(T)^e$ for some $e\geq1$ and some irreducible polynomial ${Q(T)} \in {\Q[T]}$, and
\item $Q_{<0}(T)$ is an irreducible polynomial in $\Q_{p}[T].$
\end{itemize}
\end{cond}
Recall that a \textit{$K3$ surface} $X$ over a field is a projective smooth surface with trivial canonical bundle satisfying $H^1(X, \mathscr{O}_X)=0$. It is well-known that if we write the $L$-function of a non-supersingular $K3$ surface $X$ over a finite field $\F_q$ in the following form
$$L(X/{\F_q}, T):={\rm{det}}(1-T{\mathrm{Frob}}_{\F_q}; H^2_{\rm{\acute{e}t}}(X_{\overline{\F}_q}, \Q_{\ell}{(1)}))=\prod^{22}_{i=1}(1-{\gamma_i}T),$$
then the polynomial
$$L_{\mathrm{{trc}}}(X/{\F_q}, T):=\prod_{{\gamma_i} \notin \mu_{\infty}}(1-{\gamma_i}T)$$
satisfies Condition \ref{cond}; see \cite[Theorem 1]{Taelman}. Here, ${\mathrm{Frob}}_{\F_q}\in{\rm{Gal}}({\overline{\F}_q}/{\F_q})$ is the geometric Frobenius element, and $\mu_{\infty}$ is the set of roots of unity. We call $L_{\mathrm{{trc}}}(X/{\F_q}, T)$ the \textit{transcendental part} of $L(X/{\F_q}, T)$.

In this paper, we shall prove the following conjecture due to Taelman under some mild restrictions on the characteristic.

 \begin{conj}\label{Taelman-conjecture}
Let $p$ be a prime number and $m$ an integer with $1 \leq m \leq 10$.
For a polynomial
$$L(T)=\prod^{2m}_{i=1}(1-{\gamma_i}T)\in{1+T\Q[T]}$$
satisfying Condition \ref{cond} for a power $q$ of $p$, there exist a positive integer $n\geq1$ and a $K3$ surface $X$ over $\F_{q^n}$ such that 
$$L_{\mathrm{{trc}}}(X/{\F_{q^n}}, T)=\prod^{2m}_{i=1}(1-{\gamma^n_i}T).$$ 
\end{conj}

The main result of this paper is as follows.

\begin{thm}\label{uncondition}
Let $p$ be a prime number and $m$ an integer with $1 \leq m \leq 10$.
Let
$$L(T)=\prod^{2m}_{i=1}(1-{\gamma_i}T)\in{1+T\Q[T]}$$
be a polynomial satisfying Condition \ref{cond} for a power $q$ of $p$.
\begin{enumerate}
\item When $p\geq7$, the assertion of Conjecture \ref{Taelman-conjecture} is true.
\item When $p = 5$, the assertion of Conjecture \ref{Taelman-conjecture} is true if at least one of the following conditions holds:
\begin{itemize}
\item $1\leq{m}\leq9$.
\item The discriminant $\disc(\Q(\gamma_1))$ of the number field $\Q(\gamma_1)$ is a square.
\item The integer $e$ of Condition \ref{cond} is even.
\end{itemize}
 \end{enumerate}
\end{thm} 

Taelman proved that, assuming a strong version of the existence of semistable reduction for $K3$ surfaces in the sense of \cite[Assumption $(\star)$]{Liedtke-Matsumoto}, then Conjecture \ref{Taelman-conjecture} is true \cite[Theorem 3]{Taelman}.

\begin{rem}\label{Tate}
\rm{For a $K3$ surface $X$ over $\F_{q^n}$ as in Conjecture \ref{Taelman-conjecture} or Theorem \ref{uncondition}, the height $h(X_{\overline{\F}_p})$ of the formal Brauer group associated with $X_{\overline{\F}_p}:=X\otimes_{\F_{q^n}}{\overline{\F}_p}$ is equal to the integer $h$ of Condition \ref{cond}; see \cite[Proposition 7]{Taelman}. We also have $\rho(X_{\overline{\F}_p})=22-2m$ by the Tate conjecture \cite{Madapusi}, \cite{Maulik}, \cite{Charles13}, \cite{Kim-Madapusi}. Here $\rho(X_{\overline{\F}_p})$ is the Picard number of $X_{\overline{\F}_p}$. In fact, for $K3$ surfaces $X$ constructed in the proof of Theorem \ref{uncondition}, the equality $\rho(X_{\overline{\F}_p})=22-2m$ follows from the construction, and we do not need the Tate conjecture for this equality; see Section \ref{CM} and Section \ref{main} for details.}
\end{rem} 
\begin{rem}
\rm{It seems an interesting but difficult question to ask whether we can take $n=1$ in Conjecture \ref{Taelman-conjecture} or Theorem \ref{uncondition}; see \cite[Remark 5]{Taelman}.}
\end{rem}
\begin{rem}
\rm{Currently, by our methods, it seems difficult to prove Conjecture \ref{Taelman-conjecture} unconditionally in the remaining cases (i.e.\ when $p\leq3$, or $p=5$, $m=10$, $\disc(\Q(\gamma_1))$ is not a square, and $e$ is odd). See Remark \ref{can not weaken} and Remark \ref{counter} for explanations on technical difficulties.}
\end{rem}

In the proof of Theorem \ref{uncondition}, we basically follow Taelman's strategy. We shall construct a $K3$ surface over a finite field with given $L$-function as the reduction modulo $p$ of an appropriate $K3$ surface with $CM$. To obtain unconditional results, we need to apply Matsumoto's good reduction criterion without a priori assuming a strong version of the existence of semistable reduction for $K3$ surfaces in the sense of \cite[Assumption $(\star)$]{Liedtke-Matsumoto}. The point is that we can construct an appropriate $K3$ surface with $CM$ so that it admits an elliptic fibration with a section or an ample line bundle of low degree. Then we use Saito's results to construct a strictly semistable model and apply a slight refinement of Matsumoto's criterion.
  
 The outline of this paper is as follows. In Section \ref{good}, we recall Matsumoto's  good reduction criterion for $K3$ surfaces. We also give a slight refinement for $K3$ surfaces with an elliptic fibration with a section and apply these results to $K3$ surfaces with $CM$. In Section \ref{lattice}, we recall Bayer-Fluckiger's results on quadratic spaces and $CM$ fields. Section \ref{CM} is the technical heart of this paper. In Section \ref{CM}, we construct appropriate projective $K3$ surfaces over $\C$ with $CM$. Previously, Taelman constructed for any $CM$ field $F$ with $[F:\Q]\leq20$, a projective $K3$ surface over $\C$ with $CM$ by $F$ by using Bayer-Fluckiger's results. Improving his arguments and using Nikulin's results on primitive embeddings of lattices into the $K3$ lattice, we can construct a projective $K3$ surface over $\C$ with $CM$ by $F$ so that it admits an elliptic fibration with a section or an ample line bundle of degree $2$. In Section \ref{main}, we prove Theorem \ref{uncondition}. In Section \ref{given rho and h}, as an application of Theorem \ref{uncondition}, we construct $K3$ surfaces with given geometric Picard number and height over finite fields of characteristic $p\geq5$. 
\section{Good reduction criterion for $K3$ surfaces}\label{good}
In this section, we recall the good reduction criterion for $K3$ surfaces due to Matsumoto \cite{Matsumoto}. We also give a slight refinement for $K3$ surfaces with an elliptic fibration with a section.
Then we apply these results to $K3$ surfaces with $CM$.

First, we recall well-known facts about $K3$ surfaces.
Let $U$ be the \textit{hyperbolic plane}, i.e.\ $U$ is the lattice $$U:=\Z{e} \oplus \Z{f}$$
such that its intersection pairing is given by $(e, e)=(f, f)=0$ and $(e, f)=1$.

The following result is well-known; see the proof of \cite[Chapter 14, Corollary 3.8]{Huybrechts} for example.
Although it is stated and proved only for complex $K3$ surfaces in \cite[Chapter 14, Corollary 3.8]{Huybrechts}, the same argument works for $K3$ surfaces over an algebraically closed field of characteristic different from $2, 3$.

\begin{prop}\label{elliptic}
Let $X$ be a $K3$ surface over an algebraically closed field $k$ of characteristic different from $2, 3$.
Then the following assertions are equivalent:
\begin{itemize}
	\item There exists an embedding of lattices
$$U \hookrightarrow \Pic(X),$$
where $U$ is the hyperbolic plane.
	\item $X$ admits an elliptic fibration with a section.
\end{itemize}
\end{prop}

\begin{rem}
\rm{If the assertions of Proposition \ref{elliptic} hold, the $K3$ surface $X$ is isomorphic to the Jacobian of an elliptic fibration of $X$ with a section; see \cite[Chapter 11]{Huybrechts} for details.
However, we will not use this fact in this paper.}
\end{rem}

\begin{prop}[{\cite[Chapter 14, Corollary 3.8]{Huybrechts}}]\label{12}
Let $X$ be a $K3$ surface over an algebraically closed field $k$ of characteristic different from $2, 3$. Assume that the Picard number of $X$ satisfies $12 \leq \rho(X) \leq 20$.
Then $X$ admits an elliptic fibration with a section.
\end{prop}
\begin{proof}
See \cite[Chapter 14, Corollary 3.8]{Huybrechts}. Note that although \cite[Chapter 14, Corollary 3.8]{Huybrechts} is stated and proved only for $K3$ surfaces over $\C$, it also holds in characteristic $p\geq5$. Indeed, the $K3$ surface $X$ is not supersingular by the Tate conjecture \cite{Madapusi}, \cite{Maulik}, \cite{Charles13}, \cite{Kim-Madapusi}, and we can lift the Picard group to characteristic $0$; see the proof of \cite[Proposition 5.5]{Nygaard-Ogus} and the proof of \cite[Theorem 5.6]{Nygaard-Ogus}. Then we apply Proposition \ref{elliptic}.
\end{proof}

Next, we recall Saito's results on the construction of strictly semistable models \cite{Saito}. We state his results in the form suitable for our purposes. For an extension of fields $L/K$ and a scheme $X$ over $K$, we put $X_{L}:=X\otimes_{K}{L}$. See also the proof of (c) in \cite[p.253]{Matsumoto}.
\begin{prop}[Saito {\cite[Corollary 1.9, Theorem 2.9]{Saito}}]\label{semistable model}
Let $K$ be a discrete valuation field whose residue field is of characteristic $p \geq 5$. Let $X$ be a $K3$ surface over $K$. Assume that $X$ admits an elliptic fibration with a section. Then there exist a finite extension $L/K$ and a regular scheme $\mathscr{Y}$ projective flat over the valuation ring $\O_L$ whose special fiber is a reduced simple normal crossing divisor and generic fiber is birational to $X_{L}$. 
\end{prop} 

The following theorem is essentially proved in Matsumoto's paper; see \cite[Remark 1.2 (5)]{Matsumoto} and the proof of (c) in \cite[p.253]{Matsumoto}. See also \cite{Liedtke-Matsumoto}.
\begin{thm}[Matsumoto \cite{Matsumoto}]\label{criterion}
Let $K$ be a henselian discrete valuation field with perfect residue field of characteristic $p>0$. Let $X$ be a $K3$ surface over $K$ satisfying at least one of the following conditions:
\begin{itemize}
\item The $\ell$-adic $\acute{e}$tale cohomology $H^2_{\mathrm{{\acute{e}t}}}(X_{\overline{K}}, {{\Q}_{\ell}})$ is unramified for some prime number $\ell\neq{p}$.
\item $K$ is complete of characteristic $0$ and the $p$-adic $\acute{e}$tale cohomology $H^2_{\mathrm{{\acute{e}t}}}(X_{\overline{K}}, {{\Q}_{p}})$ is crystalline.
 \end{itemize}
 Moreover, assume that at least one of the following conditions holds:
 \begin{itemize}
\item $X$ admits an ample line bundle $\mathscr{L}$ with $p>(\mathscr{L})^2+4$.
\item $p\geq5$ and $X$ admits an elliptic fibration with a section.
\end{itemize}
  Then there exist a finite extension $K'/K$ and an algebraic space $\mathscr{X}$ proper smooth over the valuation ring $\mathscr{O}_{K'}$ such that $\mathscr{X}\otimes_{\mathscr{O}_{K'}}{K'}\simeq{X_{K'}}$. 
 \end{thm}
 \begin{proof}
  When $X$ admits an ample line bundle $\mathscr{L}$ with $p>(\mathscr{L})^2+4$, the assertion is proved in \cite[Theorem 1.1]{Matsumoto}.
  Precisely, the base field $K$ is assumed to be complete in \cite{Matsumoto}. But exactly the same proof works for henselian discrete valuation fields. See also \cite{Liedtke-Matsumoto} where the results are stated over henselian discrete valuation fields.
  
   As explained in \cite[Remark 1.2 (5)]{Matsumoto} and \cite{Liedtke-Matsumoto}, if we can construct a strictly semistable model, we achieve the theorem. More precisely, if we have a finite extension $L/K$ and a regular scheme $\mathscr{Y}$ proper flat over the valuation ring $\mathscr{O}_L$ whose special fiber is a reduced simple normal crossing divisor and generic fiber is birational to $X_{L}$, then $X$ satisfies \cite[Assumption $(\star)$]{Liedtke-Matsumoto} by \cite[Proposition 3.1]{Liedtke-Matsumoto}. Here we need $p \geq 5$.
   Then, we achieve the theorem by \cite[Theorem 3.3]{Liedtke-Matsumoto}.
   
When $p \geq 5$ and $X$ admits an elliptic fibration with a section, by Proposition \ref{semistable model}, we achieve the theorem as above. 
\end{proof}

\begin{rem}
{\rm{It is not always possible to construct $\mathscr{X}$ as a scheme in Theorem \ref{criterion}; see \cite[Example 5.2]{Matsumoto}.}}
\end{rem}
We recall the definition of a complex projective $K3$ surface with $CM$. Let $X$ be a projective $K3$ surface over $\C$. Let $$T_{X}:=\Pic(X)^{\perp}_{\Q} \subset H^{2}(X, \Q(1))$$ be the transcendental part of the cohomology, which has the $\Q$-Hodge structure coming from $H^{2}(X, \Q(1))$. Let $$E_{X}:={\rm{End}}_{\rm{Hdg}}(T_{X})$$ be the $\Q$-algebra of $\Q$-linear endomorphisms on $T_{X}$ preserving the $\Q$-Hodge structure on it.
We say $X$ has \textit{complex multiplication $(CM)$} if $E_{X}$ is $CM$ and ${\rm{dim}}_{E_{X}}(T_{X})=1$.
Here a number field is called $CM$ if it is a purely imaginary quadratic extension of a totally real number field. 
Pjatecki{\u\i}-{\v{S}}apiro and {\v{S}}afarevi{\v{c}} showed that every $K3$ surface with $CM$ is defined over a number field \cite[Theorem 4]{Shafarevich2}.
Rizov showed that, for a $CM$ field $F$, a $K3$ surface $X$ over $\C$ with $CM$ by $F$ is defined over an abelian extension of $F$ \cite[Corollary 3.9.4]{Rizov10}. However we will not use this fact in this paper.

As an application of Theorem \ref{criterion}, we have the following results on the reduction of $K3$ surfaces with $CM$; see also \cite[Theorem 6.3]{Matsumoto}. 
 \begin{cor}\label{CM-reduction}
 Let $X$ be a projective $K3$ surface over $\C$ with $CM$. Let $K$ be a number field embedded in $\C$ such that $X$ is defined over $K$. Let $v$ be a finite place of $K$ above a prime number $p$. Assume that at least one of the following conditions holds:
 \begin{itemize}
\item $X$ admits an ample line bundle $\mathscr{L}$ with $p>(\mathscr{L})^2+4$.
\item $p\geq5$ and $X$ admits an elliptic fibration with a section.
\end{itemize} 
Then there exist a finite extension $K'/K_v$ and an algebraic space $\mathscr{X}$ proper smooth over the valuation ring $\mathscr{O}_{K'}$ such that $\mathscr{X}\otimes_{\mathscr{O}_{K'}}{K'}\simeq{X_{K'}}$.
Here $K_v$ is the completion of $K$ at $v$.
 \end{cor}
 \begin{proof}
 Since $X$ has $CM$, after replacing $K_v$ by a finite extension of it, we may assume $H^2_{\mathrm{{\acute{e}t}}}(X_{\overline{K}_v}, {{\Q}_{\ell}})$ is unramified for some prime number $\ell\neq{p}$; see the proof of \cite[Theorem 6.3]{Matsumoto}. Then we apply Theorem \ref{criterion}.
 \end{proof}

\section{Quadratic spaces and $CM$ fields}\label{lattice}
In this section, we recall the results of Bayer-Fluckiger on embeddings of quadratic spaces over $\Q$ arising from a $CM$ field into a given quadratic space over $\Q$.

First, we recall some definitions on quadratic spaces. Let $k$ be a field of characteristic different from $2$ and $(V, q)$ a quadratic space over $k$. This means that $V$ is a finite dimensional $k$-vector space equipped with a nondegenerate symmetric bilinear form
$$q\colon {V}  \times  {V}  \longrightarrow {k}.$$
The \textit{determinant} ${\rm{det}}(V)\in{k^*}/{({k^*}^2)}$ and the \textit{Hasse invariant} $w(V)\in\Br(k)[2]$ are defined as follows. We put $m:={\rm{dim}}_{k}V$. There exists an isomorphism of quadratic spaces 
$$V\simeq\langle{a_1, \dotsc, a_m}\rangle,$$                                  
where $a_1, \dotsc, a_m\in{k}^*$ and $\langle{a_1, \dotsc, a_m}\rangle$ is a diagonal form. Then we define
\begin{align*}
{\rm{det}}(V)&:=\prod^{m}_{i=1}a_i\in{{k^*}/{({k^*}^2)}},\\
w(V)&:=\sum_{i<j}(a_i,a_j)\in\Br(k)[2],
\end{align*}
where $(\ ,\ )$ denotes the Hilbert symbol and $\Br(k)[2]$ denotes the subgroup of elements of the Brauer group $\Br(k)$ killed by $2$. Then ${\rm{det}}(V)$ and $w(V)$ do not depend on the choice of an isomorphism $V\simeq\langle{a_1, \dotsc, a_m}\rangle$; see \cite[Chapter IV]{Serre} for details.

Let $F$ be a $CM$ field of degree $[F:\Q]=2m$.
The maximal totally real subfield of $F$ is denoted by $F_0$.
The discriminant of $F$ is denoted by $\disc(F) \in \Z$. It is easy to see that $(-1)^m\disc(F)$ is a positive integer.
For each $\lambda\in{F^*_0}$, we define the quadratic space $(F, q_{\lambda})$ over $\Q$ by  
\begin{align*}
q_{\lambda}\colon {F}  \times  {F}  &\longrightarrow {\Q}\\
                                    ({x},{y}) &\longrightarrow {{\mathrm{Tr}}_{F/{\Q}}}(\lambda{x}{i(y)}),
\end{align*}
where $i\colon F\to{F}$ is the involution defined by the nontrivial element of ${\rm{Gal}}(F/{F_0})$.

Let $(V,q)$ be a quadratic space over $\Q$ of dimension $2m$ and $F$ a $CM$ field of degree $[F:\Q]=2m$.
We say the \textit{hyperbolicity condition} holds for $F$ and $V$ if, for every prime number $p$ such that all the places of $F_0$ above $p$ split in $F$, we have 
$$w(V_{\Q_p})=w(U^{\oplus{m}}_{\Q_p}) \in \Br(\Q_p)[2],$$
where we put $V_{\Q_p}:=V\otimes_{\Q}{\Q_p}$ and $U_{\Q_p}:=U\otimes_{\Z}{\Q_p}$ is the hyperbolic plane over $\Q_p$.

In \cite{Bayer}, Bayer-Fluckiger showed general results on embeddings of an algebraic torus into a given orthogonal group over a global field of characteristic different from $2$. In terms of quadratic spaces, her results are stated as follows.
See \cite[Proposition 1.3.1 and Corollary 4.0.3]{Bayer}.
\begin{thm}[Bayer-Fluckiger {\cite[Proposition 1.3.1, Corollary 4.0.3]{Bayer}}]\label{CM-Bayer}
Let $F$ be a $CM$ field of degree $[F:\Q]=2m$, and $(V, q)$ a quadratic space over $\Q$ of dimension $2m$. Then there exists $\lambda\in{F^*_0}$ such that $$(V, q)\simeq(F, q_{\lambda})$$ if and only if all of the following conditions hold:
\begin{itemize}
\item ${\rm{det}}(V)=(-1)^{m}{\rm{disc}}(F)$ in ${{{\Q}^*}/{({{\Q}^*}^2})}$.
\item The signature of $V$ is even.
\item The hyperbolicity condition holds for $F$ and $V$.
\end{itemize}
\end{thm}

\section{Construction of $K3$ surfaces with $CM$ with additional conditions}\label{CM}
In this section, we shall construct projective $K3$ surfaces over $\C$ with $CM$ by a given $CM$ field. Here the point is that, in order to make Taelman's argument unconditional, we require the $K3$ surfaces constructed in this section satisfy some additional conditions; see Proposition \ref{degree2}.

First, we make some preparations.
\begin{lem}\label{Chebotarev}
Let $F$ be a $CM$ field with maximal totally real subfield $F_0$. Then there exist infinitely many prime numbers $p$ such that there exists a place of $F_0$ above $p$ which does not split in $F$.
\end{lem}
\begin{proof}
See the proof of \cite[Lemma 2.2.2]{Bayer}.
\end{proof}

 \begin{lem}\label{progression}
Let $p$ be an odd prime number. Let $x$ be an integer not divisible by $p$. Then there exists an odd prime number $q$ such that $q\equiv x \pmod p$ and $q\equiv3 \pmod 4$.
\end{lem}
\begin{proof}
 By the Chinese remainder theorem, there exists an integer $a\in\Z$ satisfying $a \equiv x \pmod p$ and $a \equiv 3 \pmod 4$. By Dirichlet's theorem on arithmetic progressions, there exists a prime number $q$ satisfying $q \equiv a \pmod {4p}$. The prime number $q$ is odd, and satisfies $q\equiv x \pmod p$ and $q\equiv3 \pmod 4$.\end{proof}
 
Let $U=\Z{e} \oplus \Z{f}$ be the hyperbolic plane as in Section \ref{good}. We define the \textit{$K3$ lattice}
$$\Lambda_{K3}:={(-E_8)}^{\oplus{2}}\oplus{U}^{\oplus{3}},$$
which is an even  unimodular lattice of signature $(3, 19)$ over $\Z$; see \cite[Lemma 17]{Taelman}.

\begin{lem}\label{latticeCM}
Let $m$ be an integer with $6\leq{m}\leq10$, and $F$ a $CM$ field of degree $[F:\Q]=2m$ with maximal totally real subfield $F_0$. There exist an even lattice $N$ of rank $22-2m$ and signature $(1, 21-2m)$, and a primitive embedding 
$$N \hookrightarrow \Lambda_{K3}$$
satisfying the following conditions:
\begin{itemize}
\item For the orthogonal complement $T:=N^{\perp}\subset \Lambda_{K3}$, the quadratic space $(T_{\Q}, \langle, \rangle)$ is isomorphic to $(F, q_{\lambda})$ for some ${\lambda}\in{F_0}$.
\item If $6 \leq m \leq 9$, or $m=10$ and ${\rm{disc}}(F)$ is a square, there exists an embedding $U \hookrightarrow N$.
\item If $m=10$ and ${\rm{disc}}(F)$ is not a square, there exists an element $x \in N$ with $(x)^2=2$ and no element $y \in N$ satisfies $(y)^2=-2$.
\end{itemize}
Here, an embedding $N \hookrightarrow \Lambda_{K3}$ is called {\em primitive} if the cokernel $\Lambda_{K3}/N$ is torsion-free. We put $T_{\Q}:=T\otimes_{\Z}{\Q}$ and $\langle, \rangle$ is the pairing induced from the pairing on $\Lambda_{K3}$.
\end{lem}
\begin{proof}
Since $F$ is a $CM$ field of degree $2m$, we can write ${\rm{disc}}(F)={(-1)^m}n\in \Z$ for some positive integer $n\geq1$. By Lemma \ref{Chebotarev}, there exists an odd prime number $p_1$ such that there exists a place of $F_0$ above $p_1$ which does not split in $F$. By Lemma \ref{progression}, there exists an odd prime number ${p_2}\neq{p_1}$ such that $p_2 \equiv 3 \pmod 4$ and
$$\bigg( \frac{p_2}{p_1} \bigg)=
\begin{cases}
1 &\quad  ({p_1}\equiv3\ ({\rm{mod}}\ 4))\\
-1 &\quad  ({p_1}\equiv1\ ({\rm{mod}}\ 4)),\\
\end{cases}
$$
where $\displaystyle \bigg( \frac{p_2}{p_1} \bigg)$ is the Legendre symbol.

Let us consider the following lattice $N$. When $6 \leq m \leq 9$, we put
$$N:=
\begin{cases}
U \oplus\langle{-4n}\rangle\oplus{\langle{-4}\rangle}^{\oplus{19-2m}} & (m=6, 9)\\
U \oplus\langle{-4n}\rangle\oplus{{\langle{-4}\rangle}^{\oplus2}}\oplus\langle{-4p_1}\rangle\oplus\langle{-4p_2}\rangle\oplus\langle{-4{p_1}{p_2}}\rangle & (m=7)\\
U \oplus\langle{-4n}\rangle\oplus\langle{-4p_1}\rangle\oplus\langle{-4p_2}\rangle\oplus\langle{-4{p_1}{p_2}}\rangle & (m=8).\\
\end{cases}
$$
When $m=10$ and $\disc(F)$ is a square, we put
$$N:=U.$$
When $m=10$ and $\disc(F)$ is not a square, we put
$$N:=\langle{2}\rangle\oplus\langle{-8n}\rangle.$$

The lattice $N$ is an even lattice of rank $22-2m$ and signature $(1, {21-2m})$. Since $22-2m\leq10$, we have a primitive embedding
$$N \hookrightarrow \Lambda_{K3}$$
by Nikulin's theorem \cite[Theorem 1.14.4]{Nikulin}, \cite[Chapter 14, Corollary 3.1]{Huybrechts}. Let $$T:=N^{\perp} \subset \Lambda_{K3}$$ be the orthogonal complement of $N$ in $\Lambda_{K3}$.

For the $K3$ lattice $\Lambda_{K3}$, we have: 
\begin{align*}
{\rm{sign}}(\Lambda_{K3, \Q})&=(3,\ 19),\\
{\rm{det}}(\Lambda_{K3, \Q})&=-1\in{{{\Q}^*}/{({{\Q}^*}^2})},\\
w(\Lambda_{K3, \Q})&=(-1,\ -1) \in \Br(\Q)[2].
\end{align*}
Hence we have:
\begin{align*}
{\rm{sign}}{(T_{\Q})}&=(2, 2m-2),\\
{\rm{det}}(T_{\Q})&=-{\rm{det}}(N_{\Q})=n\in{{{\Q}^*}/{({{\Q}^*}^2})},\\
w(T_{\Q})&=(-1, -1)+w(N_{\Q})+(-n, n)=(-1, -1)+w(N_{\Q})\in \Br(\Q)[2].
\end{align*}
See \cite[Lemma 14]{Taelman} for the calculation of these invariants. Here $(-n, n)\in \Br(\Q)[2]$ is trivial since the equation $-nX^2+nY^2=Z^2$ has a non-trivial solution such as $(X, Y, Z)=(1, 1, 0)$.

We shall show that the hyperbolicity condition (see Section \ref{lattice}) holds for $F$ and $T_{\Q}$. Namely, for any prime number $q$ such that all the places in $F_0$ above $q$ split in $F$, we shall show
$$w(U^{\oplus{m}}_{\Q_q})=w(T_{\Q_q}) \in \Br(\Q_q)[2] \simeq \Z/{2\Z}.$$
Note that we have $$w(U^{\oplus{m}}_{\Q_q})=((-1)^{m(m-1)/2}, -1)_q.$$
\begin{itemize}
\item When $m=6, 9$, we have 
$$N_{\Q}\simeq \langle{1}\rangle\oplus{{\langle{-1}\rangle}^{\oplus20-2m}}\oplus\langle{-n}\rangle.$$
Hence we have 
$$w(T_{\Q_q})=
\begin{cases}
(-1, -1)_q & (m=6)\\
0 & (m=9).
\end{cases}
$$ 
Hence we have $w(U^{\oplus{m}}_{\Q_q})=w(T_{\Q_q})$ for any prime number $q$.

\item When $m=7, 8$, we have
$$N_{\Q}\simeq
\begin{cases}
 \langle{1}\rangle\oplus{{\langle{-1}\rangle}^{\oplus3}}\oplus\langle{-n}\rangle\oplus\langle{-p_1}\rangle\oplus\langle{-p_2}\rangle\oplus\langle{-p_1p_2}\rangle & (m=7)\\ 
 \langle{1}\rangle\oplus{{\langle{-1}\rangle}}\oplus\langle{-n}\rangle\oplus\langle{-p_1}\rangle\oplus\langle{-p_2}\rangle\oplus\langle{-p_1p_2}\rangle& (m=8). 
\end{cases}
$$
Hence we have
$$w(T_{\Q_q})=
\begin{cases} 
(-1, -1)_{q}+(-p_1, -p_2)_{q} & (m=7)\\
(-p_1, -p_2)_{q} & (m=8).
\end{cases}$$
Let $q$ be a prime number such that all the places in $F_0$ above $q$ split in $F$. Then we have $q \neq p_1$. By the above calculation, it is enough to show $(-p_1, -p_2)_{q}=0$ for any prime number $q \neq p_1$. When $q \neq 2, p_1, p_2$, we have $(-p_1, -p_2)_{q}=0$. When $q=2$, since $-p_2 \equiv 1 \pmod{4}$, we have $(-p_1, -p_2)_{2}=0$; see \cite[Chapter III, Theorem 1]{Serre}. When $q=p_2$, it is enough to prove 
$$\bigg( \frac{-p_1}{p_2} \bigg)=1.$$
See \cite[Chapter III, Theorem 1]{Serre}. By the Quadratic Reciprocity Law \cite[Chapter I, Theorem 6]{Serre}, we have 
$$\bigg(\frac{-p_1}{p_2}\bigg)=\bigg(\frac{-1}{p_2}\bigg)\bigg(\frac{p_1}{p_2}\bigg)=-(-1)^{(p_1-1)(p_2-1)/4}\bigg(\frac{p_2}{p_1}\bigg)=1.$$

\item When $m=10$ and ${\rm{disc}}(F)$ is a square, we have
$$N_{\Q}\simeq \langle{1}\rangle\oplus{\langle{-1}\rangle}.$$ 
Hence we have
$$w(T_{\Q_q})=(-1, -1)_{q}.$$ 
Therefore we have $w(U^{\oplus{m}}_{\Q_q})=w(T_{\Q_q})$ for any prime number $q$.

\item When $m=10$ and ${\rm{disc}}(F)$ is not a square, we have
$$N_{\Q}\simeq \langle{2}\rangle\oplus{\langle{-2n}\rangle}.$$ 
Hence we have
$$w(T_{\Q_q})=(-1, -1)_{q}+(2, -2n)_{q}.$$
Let $q$ be a prime number such that all the places in $F_0$ above $q$ split in $F$. Then we have ${\rm{disc}}(F)=1\in{{{\Q}^*_q}}/({{{\Q}^{*2}_q}}).$ Since $\disc(F)=(-1)^{10}n=n$, we have $(2, -2n)_{q}=(2, -2)=0.$ Hence we have $w(U^{\oplus{10}}_{\Q_q})=(-1, -1)_q=w(T_{\Q_q}).$ 
\end{itemize}

By Theorem \ref{CM-Bayer}, the quadratic space $(T_{\Q}, \langle, \rangle)$ is isomorphic to $(F, q_{\lambda})$ for some ${\lambda}\in{F^*_0}$.
If $m=10$ and ${\rm{disc}}(F)$ is not a square, it is easy to see that no element $y\in{N}$ satisfies $(y, y)=-2$. Indeed, if there exist $a, b\in\Z$ such that
$$2a^2-8nb^2=-2,$$
then $a^2+1$ is divisible by $4$, which is absurd. 
\end{proof}

Now, we shall show the main results of this section. Compare with Taelman's construction in the proof of \cite[Theorem 4]{Taelman}.
 
\begin{prop}\label{degree2}
Let $m$ be an integer with $1\leq{m}\leq10$, and $F$ a $CM$ field of degree $[F:\Q]=2m$. Then there exists a projective $K3$ surface $Z$ over $\C$ with $CM$ by $F$ satisfying the following conditions: 
\begin{itemize}
\item $\rho(Z)=22-2m$.
\item If $1\leq{m}\leq9$, or $m=10$ and ${\rm{disc}}(F)$ is a square, the $K3$ surface $Z$ admits an elliptic fibration with a section.
\item If $m=10$ and ${\rm{disc}}(F)$ is not a square, the $K3$ surface $Z$ admits an ample line bundle of degree $2$.
\end{itemize}
\end{prop}
\begin{proof}
If $1\leq{m}\leq5$, by \cite[Theorem 4]{Taelman}, there exists a projective $K3$ surface $Z$ over $\C$ with $CM$ by $F.$ Since $\rho(Z)=22-2m\geq12$, by Proposition \ref{12}, the $K3$ surface $Z$ admits an elliptic fibration with a section.

If $6\leq m \leq 10$, by Lemma \ref{latticeCM}, there exist an even lattice $N$ of rank $22-2m$ and signature $(1, 21-2m)$, an element ${\lambda}\in{F^*_0}$, and a primitive embedding 
$$N \hookrightarrow \Lambda_{K3}$$
satisfying $(T_{\Q}, \langle, \rangle)\simeq (F, q_{\lambda})$, where we put $T:=N^{\perp} \subset \Lambda_{K3}$, $T_{\Q}:=T\otimes_{\Z}{\Q}$, and $\langle, \rangle$ is the pairing induced from the pairing on $\Lambda_{K3}$.
Following Taelman's arguments \cite[Theorem 4]{Taelman}, we shall show that there exists a projective $K3$ surface $Z$ over $\C$ with $CM$ by $F$ satisfying $$N\simeq\Pic({Z}).$$
We fix an isomorphism $(T_{\Q}, \langle, \rangle)\simeq (F, q_{\lambda})$.
Take an embedding $\epsilon \colon F \hookrightarrow \C$ with $\epsilon(\lambda)>0$.
Such an embedding exists because the signature of $T$ is $(2, 2m-2)$.
Let $\overline{\epsilon}$ be the complex conjugate of $\epsilon$.
Then we have a natural decomposition
\[
F \otimes_{\Q}\C \simeq \C_{\epsilon} \times \C_{\overline{\epsilon}} \times \prod_{f \neq \epsilon, \overline{\epsilon}} \C_f,
\]
where, for an embedding $f \colon F \hookrightarrow \C$, we denote $\C$ by $\C_f$ if we consider $\C$ as an $F$-algebra by $f$.
This decomposition gives a $\Z$-Hodge structure of weight $0$ on $\Lambda_{K3}$ such that
\[
\begin{cases}
	\Lambda_{K3}^{1, -1}=\C_{\epsilon}, \\
	\Lambda_{K3}^{-1, 1}=\C_{\overline{\epsilon}}, \\
	\Lambda_{K3}^{0, 0}=N_{\C} \oplus \prod_{f \neq \epsilon, \overline{\epsilon}} \C_f.
\end{cases}
\]
By the surjectivity of the period map for $K3$ surfaces over $\C$ \cite[Theorem 1]{Todorov},
there is a complex analytic $K3$ surface $Z$ over $\C$ and a Hodge isometry
\[
\Lambda_{K3} \simeq H^2(Z, \Z(1)).
\]
See also \cite[Chapter 6, Remark 3.3]{Huybrechts}.
Since $F \cap \prod_{f \neq \epsilon, \overline{\epsilon}} \C_f=0$, a natural embedding
\[
N \hookrightarrow \Lambda_{K3}^{0, 0} \cap \Lambda_{K3}
\]
is an isomorphism after tensoring with $\otimes_{\Z}\Q$.
Since $N \hookrightarrow \Lambda_{K3}
$ is a primitive embedding, it follows that
\[
N \simeq \Lambda_{K3}^{0, 0} \cap \Lambda_{K3}.
\]
By \cite[Chapter IV, Theorem 6.2]{BHPV}, the complex analytic $K3$ surface $Z$ is projective.
Now, the claim follows from the Lefschetz theorem on $(1, 1)$-classes.

 Moreover, if $6\leq m \leq 9$, or $m=10$ and ${\rm{disc}}(F)$ is a square, there exists an embedding 
$$U\hookrightarrow N \simeq \Pic(Z).$$
By Proposition \ref{elliptic}, the $K3$ surface $Z$ admits an elliptic fibration with a section.

If $m=10$ and ${\rm{disc}}(F)$ is not a square, no element $y\in{N} \simeq \Pic(X)$ satisfies $(y)^2=-2$, hence the $K3$ surface $Z$ does not contain smooth rational curves. Therefore, any line bundle on $Z$ in the positive cone of $\Pic(Z)_{\R}$ is ample; see \cite[Chapter 2, Proposition 1.4]{Huybrechts}. In particular, for any element $x \in N \simeq \Pic(Z)$ with $(x)^2=2$, we see that $x$ or $-x$ corresponds to an ample line bundle of degree $2$. 
 \end{proof}
The statement of Proposition \ref{degree2} on the existence of an elliptic fibration with a section on a $K3$ surface with $CM$ is optimal. Indeed, we have the following proposition.
\begin{prop}\label{degree20}
Let $F$ be a $CM$ field with $[F:\Q]=20$. Then the following assertions are equivalent: 
\begin{itemize}
\item $\disc(F)$ is a square.
\item There exists a projective $K3$ surface $Z$ over $\C$ with $CM$ by $F$ which admits an elliptic fibration with a section.
\end{itemize}
\end{prop}
\begin{proof}
If $\disc(F)$ is a square, by Proposition \ref{degree2}, there exists a projective $K3$ surface $Z$ over $\C$ with $CM$ by $F$ which admits an elliptic fibration with a section.

Next, we assume that there exists a projective $K3$ surface $Z$ over $\C$ with $CM$ by $F$ which admits an elliptic fibration with a section. Let $F_0$ be the maximal totally real subfield of $F$. It is well-known that the quadratic space $(T_{Z}, \langle, \rangle)$ is isomorphic to $(F, q_{\lambda})$ for some element $\lambda \in F^*_0$; see \cite[Proposition 1.3.10]{Chen} for example.
Here, the lattice $T_{Z}$ is the transcendental part of $H^{2}(Z, \Q(1))$ and $\langle, \rangle$ is the intersection pairing; see Section \ref{good}. We briefly recall the argument. Since ${\rm{dim}}_{F}(T_{Z})=1$, there exists an element $\alpha \in T_{Z}$ with $T_{Z}\simeq F\alpha$. Consider the bilinear form
$$\Phi \colon T_{Z} \times T_{Z}  \longrightarrow F$$
defined by 
$$\langle ex, y \rangle = {\rm{Tr}}_{F/{\Q}}(e\Phi(x, y)),$$
for all $e\in F$ and $x, y \in T_{Z}$.
We put $\lambda:=\Phi(\alpha, \alpha)$. Then we have 
$$\langle a\alpha, b\alpha \rangle=\langle ai(b)\alpha, \alpha \rangle={\rm{Tr}}_{F/\Q}(ai(b)\lambda)=q_{\lambda}(a, b)$$
for any $a, b \in F$; see \cite[Section 1.5]{Zarhin} for the first equality.
We also have 
$$\langle a\alpha, b\alpha \rangle=\langle b\alpha, a\alpha \rangle={\rm{Tr}}_{F/\Q}(bi(a)\lambda)={\rm{Tr}}_{F/\Q}(ai(b)i(\lambda)).$$
Hence $\lambda \in F^*_0$, and we have $(T_{Z}, \langle, \rangle)\simeq (F, q_{\lambda})$. Since $\rho(Z)=2$ and $Z$ has an elliptic fibration with a section, we have $\Pic(Z) \simeq U$. In particular, we have ${\rm{det}}(\Pic(Z)_{\Q})=-1$. 
Since ${\rm{det}}(\Lambda_{K3, \Q})=-1 \in {{{\Q}^*}/{({{\Q}^*}^2}})$ and $\Lambda_{K3, \Q} \simeq \Pic(Z)_{\Q}\oplus T_{Z}$, we have $$\disc(F)=(-1)^{10}{\rm{det}}(T_{Z})=1 \in {{{\Q}^*}/{({{\Q}^*}^2}})$$
by \cite[Lemma 1.3.2]{Bayer}.
\end{proof}
\section{Proof of Theorem \ref{uncondition}}\label{main}
In this section, we shall prove Theorem \ref{uncondition}.
First, we prove the following lemma. 
\begin{lem}\label{disc}
For a finite extension of $CM$ fields $F/E$ with $[F:E]=e$, we have 
$${\rm{disc}}(F)=\disc(E)^e \in {{{\Q}^*}/{({{\Q}^*}^2})}.$$
\end{lem}
\begin{proof}
We put $[E:\Q]=2m$.
Then the sign of $\disc(E)$ is $(-1)^m$.
Similarly, the sign of $\disc(F)$ is $(-1)^{em}$.
Hence $\disc(F)$ and $\disc(E)^e$ have the same sign.
Therefore, to prove Lemma \ref{disc}, it suffices to prove that 
$\disc(F)\Z$ is equal to $\disc(E)^e\Z $ in the group of fractional ideals of $\Q$ modulo squares. 

For a number field $K$,
we denote the group of fractional ideals of $K$ by $I_K$.
For an extension of number fields $K_2/K_1$, we denote the relative discriminant of $K_2/K_1$ by $\mathfrak{d}_{K_2/K_1}$. 
Let $F_0$ (resp.\ $E_0$) be the maximal totally real subfield of $F$ (resp.\ $E$).
By \cite[Chapter 3, Proposition 8]{Serre2}, we have the following equalities in $I_{\Q}/I^2_{\Q}$
\begin{align*}
\disc(F)\Z &={\mathrm{N}}_{F_0/\Q}(\mathfrak{d}_{F/F_0}),\\
\disc(E)\Z &={\mathrm{N}}_{E_0/\Q}(\mathfrak{d}_{E/E_0}).
\end{align*} 
We put $E=E_0(\sqrt{\beta})$ for some $\beta \in E_0$.
Then we have $\mathfrak{d}_{E/E_0} = \beta \mathscr{O}_{E_0}$ in $I_{E_0}/I^2_{E_0}$; see \cite[Chapter 3, Proposition 5]{Serre2}. Here $\mathscr{O}_{E_0}$ is the ring of integers of $E_0$.
We have ${\mathrm{N}}_{E_0/\Q}(\mathfrak{d}_{E/E_0})={\mathrm{N}}_{E_0/\Q}(\beta)\Z$ in $I_{\Q}/I^2_{\Q}$. 
Similarly, since $F = F_0(\sqrt{\beta})$,
we have $\mathfrak{d}_{F/F_0} = \beta \mathscr{O}_{F_0}$ in $I_{F_0}/I^2_{F_0}$ and we have ${\rm{N}}_{F_0/\Q}(\mathfrak{d}_{F/F_0})={\rm{N}}_{E_0/\Q}(\beta)^e\Z$ in $I_{\Q}/I^2_{\Q}$.
Therefore, we have the following equality in $I_{\Q}/I^2_{\Q}$
$$ \disc(F) \Z = \disc(E)^e\Z.$$
\end{proof}
\begin{proof}[\bf{Proof of Theorem \ref{uncondition}}]
We briefly recall Taelman's argument; see \cite[Section 4]{Taelman} for details. Consider a polynomial $$L(T)=\prod^{2m}_{i=1}(1-{\gamma_i}T) \in{1+T\Q[T]}$$ of degree $2m$ satisfying Condition \ref{cond}. We put $E:={\Q}(\gamma_1)$. By Condition \ref{cond}, the number field $E$ is a $CM$ field which has a unique finite place $v$ satisfying $v(\gamma_1)<0$ above $p$. Take a finite extension $F/E$ of $CM$ fields such that $[F:{\mathbb{Q}}]=2m$ and $F$ has a unique finite place above $v$; see \cite[Lemma 23]{Taelman}. By Lemma \ref{disc}, $\disc(F)$ is a square if $\disc(E)$ is a square or $e$ is even.
 
Assume that at least one of the following conditions is satisfied:
\begin{itemize}
\item $p\geq5$ and $1\leq m \leq 9$.
\item $p\geq5$, $m=10$, and $\disc(E)$ is a square.
\item $p\geq5$, $m=10$, and $e$ is even.
\end{itemize}
Then, by Proposition \ref{degree2}, there exists a projective $K3$ surface $Z$ over $\C$ with $CM$ by $F$ which admits an elliptic fibration with a section.

Assume $p\geq7$, $m=10$, $\disc(E)$ is not a square, and $e$ is odd. By Proposition \ref{degree2}, there exists a projective $K3$ surface $Z$ over $\C$ with $CM$ by $F$ which admits an ample line bundle $\mathscr{L}$ of degree $2$.
 
 In any of these cases, by the theorem of Pjatecki{\u\i}-{\v{S}}apiro and {\v{S}}afarevi{\v{c}} \cite[Theorem 4]{Shafarevich2}, the $K3$ surface $Z$ is defined over a finite extension $L$ of $F$. By Corollary \ref{CM-reduction}, after replacing $L$ by a finite extension of it, there is an algebraic space $\mathscr{Z}$ which is proper and smooth over the valuation ring $\mathscr{O}_{L_{v'}}$, where $v'$ is a finite place of $L$ above $v$. (Here, we use the inequality $p>(\mathscr{L})^2+4=6$ when $p\geq7$, $m=10$, $\disc(E)$ is not a square, and $e$ is odd.) We denote the special fiber of $\mathscr{Z}$ by $\mathscr{Z}_{s}$, which is a $K3$ surface over the residue field $k(s)$. 
 
  Rizov also proved the main theorem of complex multiplication for $K3$ surfaces \cite[Corollary 3.9.2]{Rizov10}. Taelman used Rizov's results to calculate the zeta function of the reduction modulo $p$ of a $K3$ surface with $CM$ if it has good reduction; see the proof of \cite[Proposition 25]{Taelman}. As explained in \cite[Section 4]{Taelman}, it follows that there exists a finite extension $k'/k(s)$ such that the transcendental part of the $L$-function of $\mathscr{Z}_{s}\otimes_{k(s)}{k'}$ satisfies
$$L_{\mathrm{trc}}((\mathscr{Z}_{s}\otimes_{k(s)}{k'})/{k'}, T)=\prod^{2m}_{i=1}(1-{\gamma^{[k' : \F_{q}]}_i}T).$$

 The proof of Theorem \ref{uncondition} is complete.
 \end{proof}
\begin{rem}\label{can not weaken}
\rm{It seems difficult to weaken the assumption of Theorem \ref{uncondition} by our methods. Currently, in order to apply Matsumoto's good reduction criterion for $K3$ surfaces or its variants, we always need $p\geq5$; see Theorem \ref{criterion} and \cite{Matsumoto}, \cite{Liedtke-Matsumoto}. The main reason is, in Matsumoto's proof, Saito's construction of strictly semistable models \cite{Saito} and Kawamata's results on minimal models \cite{Kawamata}, \cite{Kawamata2} are crucially used, and both results require $p\geq5$. (See \cite[Section 3]{Kawamata2} for explanations why $p\geq5$ is necessary to use Kawamata's results.) When $p=5$, we cannot treat the remaining case (i.e.\ when $m=10$, $\disc(E)$ is not a square, and $e$ is odd) by our methods. Indeed, in this case, for any $CM$ field $F$ containing $E$ with $[F:E]=e$, the discriminant $\disc(F)$ is not a square by Lemma \ref{disc}. By Proposition \ref{degree20}, no $K3$ surface with $CM$ by $F$ admits an elliptic fibration with a section. Hence we cannot apply Theorem \ref{criterion} for such $K3$ surfaces. See also Remark \ref{counter}.}
\end{rem}

\section{Construction of K3 surfaces over finite fields with given geometric Picard number and height}\label{given rho and h}
In \cite{Artin74}, Artin proved that, for a $K3$ surface $X$ over an algebraically closed field of characteristic $p>0$, the Picard number $\rho(X)$ and the height $h(X)$ of the formal Brauer group satisfy the inequality $$\rho(X)\leq22-2h(X).$$

 In this section, as an application of Theorem \ref{uncondition}, we shall construct $K3$ surfaces with given geometric Picard number and height over finite fields of characteristic $p\geq5$.
 
\begin{lem}\label{poly}
Let $p$ be a prime number, and $h', m'$ integers with $1 \leq h' \leq  m' \leq 10$. Then there exists a polynomial $$L(T) \in 1+T\Q[T]$$ of degree $2m'$ satisfying Condition \ref{cond} for a power $q$ of $p$, $h=h'$, and $m=m'$.
\end{lem}
\begin{proof}
Let $S \subset \C$ be the set of roots of unity whose minimal polynomials over $\Q$ have degree less than or equal to $20$. We put 
$$S' := \lbrace\, x \in \C \mid x=t+1/t\ {\rm{for\ some}}\ t \in S\,\rbrace.$$

We can take a monic polynomial $$F(T)=T^{m'}+c_{m'-1}T^{m'-1}+\cdots+c_{1}T+c_0 \in \Q[T]$$ 
of degree $m'$ with roots $\alpha_1,\dotsc,\alpha_{m'}\in \overline{\Q} \subset \C$ such that all of the  following conditions are satisfied:
\begin{itemize}
\item $\alpha_1,\dotsc,\alpha_{m'}\in \R$.
\item $\vert\alpha_i\vert<2$.
\item $\alpha_i \notin S'$ for any $1 \leq i \leq m'$.
\item $F(T) \in \Z_{\ell}[T]$ for all prime numbers ${\ell}\neq{p}$. 
\item $\nu_p(c_{m'-h'})=-a$ for some positive integer $a\geq1$ which is coprime to $h'$.
\item $c_{i} \in \Z$ for any $i \neq m'-h'$.
\end{itemize}
For example, we can find such polynomials as follows. We put:
\begin{itemize}
\item $f_0(T):=T-1$,
\item $f_i(T):=T^2-i$ \quad $(i=1, 2, 3),$
\item $f_4(T):=T^3-3T+1,$
\item $f_5(T):=T^4-4T^2+1.$
\end{itemize}
As a product of $f_0,\dotsc, f_5$, we can find a monic  polynomial $F_0(T) \in \Z[T]$ of degree $m'$ which has $m'$ distinct real roots $\beta_1, \dotsc, \beta_{m'} \neq 0$ with $\vert\beta_i\vert<2$ for any $i$ (for example, take $F_0(T)$ as in Table $1$ below). 
\begin{center}
\begin{tabular}{|c|c|} \hline
\ $m'$ \ \    &   \  $F_0(T)$   \ \    \\ 
\hline
  1     &    $f_0(T)$  \\
  2    &   $ f_1(T)$\\
  3   &   $f_4(T)$ \\
  4    & $ f_5(T)$\\
  5   &   $f_1(T)f_4(T)$\\
  6    &   $f_1(T)f_5(T)$\\
  7    &   $ f_4(T)f_5(T)$\\
  8    &    $ f_1(T)f_2(T)f_5(T)$\\
  9   &   $ f_1(T)f_4(T)f_5(T)$\\
  10   &  $f_1(T)f_2(T)f_3(T)f_5(T)$\\ \hline
  \end{tabular}
  \\*[2mm]
  {\sc Table 1}.
  \end{center}
  Since $S'$ is a finite set, if we take a sufficiently large positive integer $a$ which is coprime to $h'$, the polynomial 
$$F(T)=F_0(T)+{p^{-a}}T^{m'-h'}$$
satisfies the above conditions.

Then the polynomial $$L(T)=T^{m'}F(T+1/T)$$ satisfies Condition \ref{cond}. The first three conditions can be checked easily. To check the last three conditions, it is enough to show that $L(T)$ is irreducible in $\Q[T]$ and that $L_{<0}(T)$ is irreducible in $\Q_p[T]$.

If we write $L(T)$ in the following form 
$$L(T)=T^{2m'}+d_{2m'-1}T^{2m'-1}+\cdots+d_{1}T+1,$$
we have
$$
\begin{cases}
\nu_{p}(d_i)\geq 0  &\quad \ (1\leq{i}\leq{h'-1},\ {2m'-h'+1}\leq{i}\leq{2m'-1}),\\
\nu_{p}(d_i)=-a &\quad \  (i=h', 2m'-h'),\\
\nu_{p}(d_i)\geq{-a}  &\quad  \  ({h'+1}\leq{i}\leq{2m'-h'-1}).\\
\end{cases}
$$
It follows that the degree of $L_{<0}(T) \in \Q_p[T]$ is $h'$ and the inverse $\gamma$ of every root of $L_{<0}(T)$ satisfies $\nu_p(\gamma)=-a/h'.$ Hence $L(T)$ satisfies the third condition of Condition \ref{cond} for $q=p^a$ and $h=h'$. Since $a$ is coprime to $h'$, it follows that $L_{<0}(T)$ is irreducible in $\Q_p[T]$.

Assume that $L(T)$ is not irreducible in $\Q[T]$. Since $L_{<0}(T)$ is irreducible in $\Q_p[T]$, $L(T)$ is divisible by two different irreducible monic polynomials $L_1(T), L_2(T) \in \Q[T]$. Since they do not have common roots and $L_{<0}(T)$ is irreducible in $\Q_p[T]$, we may assume that every root $\gamma$ of $L_1(T)$ satisfies $\nu_p(\gamma)=0$. Since $L(T) \in \Z_{\ell}[T]$
for all prime numbers ${\ell}\neq{p}$, it follows that $L_1(T) \in \Z[T]$. Since every root of $L(T)$ has complex absolute value one, it follows that $L_1(T)$ is a cyclotomic polynomial by Kronecker's theorem. This contradicts to our construction of $L(T)$. Hence $L(T)$ is irreducible in $\Q[T]$. 
\end{proof}

\begin{lem}\label{poly2}
Let $p$ be a prime number, and $h'$ an integer with $1\leq h' \leq 10$. 
Then there exists a polynomial $$L(T)=\prod^{20}_{i=1}(1-{\gamma_i}T) \in 1+T\Q[T]$$ of degree $20$ such that Condition \ref{cond} is satisfied for a power $q$ of $p$, $h=h'$, and $m=10$, and at least one of the following conditions is satisfied:
\begin{itemize}
\item $\disc(\Q(\gamma_1))$ is a square.
\item The integer $e$ of Condition \ref{cond} is even.
\end{itemize}
\end{lem}
\begin{proof}
First, we assume $h'$ is even. We write $h'=2h''$. By Lemma \ref{poly}, there exists a polynomial $L(T) \in 1+T\Q[T]$ of degree $10$ satisfying Condition \ref{cond} for a power $q$ of $p$, $h=h''$, and $m=5$. Then the square $L(T)^2$ satisfies our conditions for $q^2$, $h=h'=2h''$, and $m=10$.

Next, we assume $h'=1, 3, 7$, or $9$. We take a totally real field $M_0$ with $[M_0:\Q]=10$ such that $M_0$ has a finite place $v$ above $p$ with ramification index $e_v=h'$ and residue degree $f_v=1$. For example, we can take $M_0:=\Q[T]/(F(T))$ for the polynomial $F(T)$ constructed in the proof of Lemma \ref{poly}. Choose a positive integer $n\geq1$ such that $v$ splits into the $CM$ field $M:=M_0(\sqrt{-n})$. As in the proof of Lemma \ref{disc}, we have 
$$\disc(M)={\rm{N}}_{M_0/\Q}(-n)=(-n)^{10} \in {{{\Q}^*}/{({{\Q}^*}^2})}.$$ 
Hence $\disc(M)$ is a square. Let $v_1, v_2$ be the finite places of $M$ above $v$. Since the class number of $M$ is finite, there exists an element $\alpha \in M^*$ with $v_1(\alpha)>0$ and $v'(\alpha)=0$ for every finite place $v' \neq v_1$ of $M$. Let $i$ be the nontrivial element of ${\rm{Gal}}(M/M_0)$. Let $\beta:=\alpha/i(\alpha)$. We shall show $\Q(\beta)=M$ by a similar argument as in the proof of \cite[Lemma D.3]{Drinfeld}. Let $u$ be a place of $\Q(\beta)$ under $v_1$. From our construction, $v_1$ is the unique place above $u$. Since $e_{v_1}=h'$ and $f_{v_1}=1$, the extension degree $[M:\Q(\beta)]$ divides $h'$. Since $h'$ is coprime to $[M:\Q]=20$, we have $[M:\Q(\beta)]=1$. Hence the minimal polynomial $L(T)$ of $\beta$ over $\Q$ has degree $20$, and it satisfies our conditions for a power $q$ of $p$, $h=h'$, and $m=10$.

Finally, we assume $h'=5$. By the above discussion, we can construct a polynomial $$L(T)=\prod^{4}_{i=1}(1-{\gamma_i}T) \in 1+T\Q[T]$$ of degree $4$ satisfying Condition \ref{cond} for a power $q$ of $p$, $h=1$, and $m=2$ such that $\disc(\Q(\gamma_1))$ is a square. Then the fifth power $L(T)^5$ satisfies our conditions for $q^5$, $h=5$, and $m=10$.
      
\end{proof}
\begin{rem}\label{counter}
\rm{For $p=5$, there exists a polynomial 
$$L(T)=\prod^{20}_{i=1}(1-{\gamma_i}T) \in 1+T\Q[T]$$
 of degree $20$ satisfying Condition \ref{cond} for a power $q$ of $5$, $h=1$, and $m=10$ such that $\disc(\Q(\gamma_1))$ is not  a square. Then, as explained in Remark \ref{can not weaken}, we cannot show the assertion of Theorem \ref{uncondition} for $L(T)$ by our methods. For example, such a polynomial $L(T)$ can be constructed as follows. Let $M_0$ be a totally real field with $[M_0:\Q]=10$ such that $5$ splits completely in $M_0$. Let $v_1,\dotsc, v_{10}$ be the places of $M_0$ above $5$. We take a $CM$ field $M/M_0$ such that $[M:M_0]=2$, $v_1$ splits completely in $M$, $v_2$ is ramified in $M$, $v_i$ is unramified in $M$ for $3\leq i \leq 10$. Then $\disc(M)$ is not a square since $\nu_5(\disc(M))=1$. Let $v'_1, v'_2$ be the finite places of $M$ above $v_1$. Take an element $\alpha \in M^*$ with $v'_1(\alpha)>0$ and $v'(\alpha)=0$ for every finite place $v' \neq v'_1$ of $M$. We put $\beta:=\alpha/i(\alpha)$, where $i$ is the nontrivial element of ${\rm{Gal}}(M/M_0)$. Then, by the same argument as in the proof of Lemma \ref{poly2}, we have $\Q(\beta)=M$, and the minimal polynomial of $\beta$ over $\Q$ satisfies the above conditions.}
\end{rem}

Recall that the Picard number $\rho(X)$ of a $K3$ surface $X$ over $\overline{\F}_p$ is a positive even integer by the Tate conjecture \cite{Madapusi}, \cite{Maulik}, \cite{Charles13}, \cite{Kim-Madapusi}; see \cite[Chapter 17, Corollary 2.9]{Huybrechts}. 
\begin{thm}\label{existence}
Let $p$ be a prime number with $p\geq5$. Let $\rho \in (2\Z)_{>0}$ be a positive even integer and $h'\in\Z_{>0}$ a positive integer. 
Then the following assertions are equivalent:
\begin{enumerate}
	\item $\rho$ and $h'$ satisfy $\rho \leq 22-2h'.$
	\item There exists a $K3$ surface $X$ over $\overline{\F}_p$ such that $\rho(X)=\rho$ and $h(X)=h'$.
\end{enumerate} 
\end{thm}
\begin{proof}
(2) $\Rightarrow$ (1): This follows from \cite[Theorem 0.1]{Artin74}.

(1) $\Rightarrow$ (2):
We put $m:=11-\rho/2$. When $\rho \geq 4$, by Lemma \ref{poly}, there exists a polynomial
$$L(T)=\prod^{2m}_{i=1}(1-{\gamma_i}T)\in{1+T\Q[T]},$$
satisfying Condition \ref{cond} for a power $q$ of $p$ and $h=h'$. When $\rho=2$, by Lemma \ref{poly2}, we may assume $\disc(\Q(\gamma_1))$ is a square or $e$ is even.

By Theorem \ref{uncondition}, there exist a positive integer $n\geq1$ and a $K3$ surface $X_0$ over $\F_{q^n}$ such that 
$$L_{\mathrm{{trc}}}(X_0/{\F_{q^n}}, T)=\prod^{2m}_{i=1}(1-{\gamma^n_i}T).$$
The $K3$ surface $X:={X_0}\otimes_{\F_{q^n}}{\overline{\F}_p}$ satisfies $\rho(X)=\rho$ and $h(X)=h'$ as required; see Remark \ref{Tate}.
\end{proof}

\subsection*{Acknowledgements}
The author is deeply grateful to my advisor, Tetsushi Ito, for his kindness, support, and advice. He gave me a lot of invaluable suggestions. The author also would like to thank Yuya Matsumoto and Seidai Yasuda for helpful suggestions and comments.
Moreover the author would like to thank the anonymous referees for sincere remarks and comments.
The work of the author was supported by JSPS Research Fellowships for Young Scientists KAKENHI Grant Number 18J22191.

\end{document}